\documentclass[11pt,english,a4paper]{smfart}
\usepackage[english]{babel}
\usepackage{amssymb,xspace}
\usepackage{amstext}
\usepackage{smfthm}
\theoremstyle{plain}
\usepackage{amsbsy,amssymb,amsfonts,latexsym}

\usepackage{enumerate}
\usepackage{graphics}
\usepackage{pstricks}
\usepackage{pst-node}

\date{\today}

\title[Frequent hypercyclicity and unconditional Schauder decompositions]{Frequent
hypercyclicity, chaos, and unconditional Schauder decompositions}

\author{Manuel De la Rosa}
\address{Institut Universitari de Matem\`{a}tica Pura i Aplicada (IUMPA),
Universitat Polit\`{e}cnica de Val\`{e}ncia, 46022 Val\`{e}ncia, Spain}
\email{madela19@upvnet.upv.es}

\author{Leonhard Frerick}
\address{FB IV Mathematik,
Universit\"{a}t Trier,
D–54286 Trier,
Germany}
\email{frerick@uni-trier.de}

\author{Sophie Grivaux}
\address{Laboratoire Paul Painlev\' e, UMR 8524, Universit\'e
Lille 1, Cit\' e Scientifique, 59655 Villeneuve d'Ascq
Cedex, France}
\email{grivaux@math.univ-lille1.fr}

\author{Alfredo Peris}
\address{IUMPA, Universitat Polit\`{e}cnica de Val\`{e}ncia, Departament
de Matem\`{a}tica Aplicada, Edifici 7A, 46022 Val\`{e}ncia, Spain}
\email{aperis@mat.upv.es}

\subjclass{47A16, 37A05, 47A35, 46B09, 46B15}

\keywords{Linear dynamical systems; hypercyclic, chaotic, and frequently
hypercyclic operators;  spaces with an unconditional Schauder decomposition}
\thanks{This work  was partially supported by ANR-Projet Blanc DYNOP, by the MEC and FEDER Projects
MTM2007-64222 and MTM2010-14909,  and by Generalitat Valenciana Project
PROMETEO/2008/101.}

\def\T{\ensuremath{\mathbb T}}
\def\R{\ensuremath{\mathbb R}}

\def\N{\ensuremath{\mathbb N}}

\newcommand{\wt}{\widetilde}

\newcommand{\spa}{\operatorname{span}}

\newcommand{\sep}{separable}

\newcommand{\hy}{hypercyclic}
\newcommand{\fhy}{frequently hypercyclic}
\newcommand{\ops}{operators}
\newcommand{\op}{operator}

\newcommand{\erg}{ergodic}
\newcommand{\eve}{eigenvector}
\newcommand{\eva}{eigenvalue}
\newcommand{\ps}{perfectly spanning set of eigenvectors associated to unimodular
eigenvalues}
\newcommand{\wrt}{with respect to}

\newcommand{\ga}{Gaussian}

\newcommand{\inv}{invariant}

\newcommand{\mea}{measure}
\newcommand{\nd}{non-degenerate}

\newcommand{\ifff}{if and only if}

\newcommand{\pss}[2]{\ensuremath{{\langle #1,#2\rangle}}}

\newcommand{\norm}[1]{\left\Vert#1\right\Vert}
\newcommand{\abs}[1]{\left\vert#1\right\vert}

\newtheorem{theorem}{Theorem}[section]

{\theoremstyle{definition}}

{\theoremstyle{definition}}

{\theoremstyle{definition}}

{\theoremstyle{definition}\newtheorem{definition}[theorem]{Definition}}

{\theoremstyle{definition}}

{\theoremstyle{definition}}

\newtheorem{question}[theorem]{Question}

{\theoremstyle{definition}\newtheorem*{FFC Criterion}{Frequent
Faber-hypercyclicity Criterion}}
\newtheorem*{Hypercyclicity Criterion}{Hypercyclicity Criterion}
{\theoremstyle{definition}\newtheorem*{GS Criterion}{Godefroy-Shapiro
Criterion}}
\def\piednote#1{\let\oldfn=\thefootnote\def\thefootnote{}\footnote{\noindent#1}%
\addtocounter{footnote}{-1}\def\thefootnote{\oldfn}}

\begin{document}

\begin{abstract}
We prove that if $X$ is any complex separable infinite-dimensional Banach space with an unconditional Schauder decomposition, $X$ supports an operator
$T$ which is chaotic and frequently
 hypercyclic.
In contrast with the complex case, we observe that there are real Banach spaces with an unconditional basis which support no chaotic operator.
\end{abstract}
\maketitle

\section{Introduction}
We are interested in this paper in the dynamics of continuous linear \ops\ acting on a complex infinite-dimensional \sep\ Banach space $X$. If $T$ is
such an \op\ on $X$, $T$ is said to be \hy\ if there exists a vector $x\in X$ (a \hy\ vector for $T$) such that $\mathcal{O}\textrm{rb}(x,T)=\{T^{n}x
\textrm{ ; } n\geq 0\}$ is dense in $X$. Hypercyclicity has had many developments in the past years, and we refer the reader
 to the recent books \cite{BM} and \cite{Grosse-ErdmannPeris10book} for a
 thorough account of the subject.\par\smallskip
We study here reinforcements of hypercyclicity: chaotic \ops\ are topologically transitive \ops\ (or, in other words, hypercyclic operators) which have
a dense set of periodic points (a vector $x\in X$ is said to be periodic if there exists an integer $N\geq 1$ such that $T^{N}x=x$). This notion of
chaos coincides in our setting with the classical one introduced by Devaney. Another strengthening of hypercyclicity is the notion of frequent
hypercyclicity, which was introduced in \cite{BG1}: $T$ is said to be \fhy\ if there exists a vector $x\in X$ such that for every non-empty open subset
$U$ of $X$, the set $\{n\geq 0 \textrm{ ; } T^{n}x\in U\}$ of instants when the iterates of $x$ under $T$ visit $U$ has positive lower density:
$$\liminf_{N\rightarrow +\infty }\frac{1}{N}\abs{\{n\leq N \textrm{ ; } T^{n}x\in U\}}>0, $$
\par\smallskip
where $\abs{A}$ denotes the cardinality of a finite subset $A\subset \N$. It was proved independently by Ansari \cite{A} and Bernal-Gonzalez \cite{B}
that any \sep\ infinite-dimensional Banach space $X$ supports a \hy\ \op. These \ops\ are of the form $T=I+K$, where $K$ is a nuclear backward weighted
shift \wrt\ a biorthogonal system of $X$. In particular the spectrum of $T$ is reduced to the point $\{1\}$. The fact that chaos and frequent
hypercyclicity are really stronger notions than \hy ity is attested by the fact that the corresponding existence result does not hold true anymore.
That is, it was shown by Bonet, Mart\'{\i}nez-Gim\'{e}nez and Peris \cite{BFP} that some Banach spaces do not support any chaotic \op, and by Shkarin
\cite{S} that the same spaces do not support any \fhy\ \op. The class of spaces considered in \cite{BFP} and \cite{S} is the class of complex
hereditarily indecomposable Banach spaces (like the space of Gowers and Maurey \cite{GM}). We recall that a Banach space $X$ is said to be hereditarily
indecomposable if no closed subspace of $X$ is decomposable  as a direct sum of infinite-dimensional subspaces. On such
 spaces every \op\ has the form $T=\lambda I+S$, where $\lambda $ is a
 scalar and $S$ is a strictly singular \op\ on $X$. Hence if $T$ is \hy,
 the spectrum of $T$ is reduced to the point $\{\lambda \}$ with $|\lambda |=1$.
 Now it is proved in \cite{S} that the spectrum of a \fhy\ \op\ cannot have an
 isolated point, and the same holds true for a chaotic \op.

The main purpose of this work is to investigate the following question:

\begin{question}\label{q4}
 Is it possible to characterize the complex \sep\ Banach spaces which
 support a \fhy\ (respectively a chaotic) \op?
\end{question}

We are not able to answer completely Question \ref{q4}, but we prove the following theorem, which gives a fairly large class of spaces on which such
\ops\ can indeed be constructed:

\begin{theorem}\label{th1}
 Let $X$ be a complex \sep\ Banach space having an unconditional Schauder
 decomposition. Then $X$ supports an \op\ which is \fhy\ and chaotic.
 \end{theorem}

In particular any complex Banach space with an unconditional basis admits a \fhy\ and chaotic \op. Actually, it is still an open question to know
whether every chaotic \op\ on a Banach space is automatically \fhy; see \cite{Grivaux10} for more details on this question.
\par\smallskip
The \fhy\ \ops\ constructed in the proof of Theorem \ref{th1} have an interesting property: they are compact (even nuclear) perturbations of diagonal
\ops\ whose diagonal coefficients are complex numbers of modulus $1$. The proof of Theorem \ref{th1} is done via a transference argument. In other
words, we first construct a class of nuclear perturbations of diagonal \ops\ on a Hilbert space, then transfer these operators to our Banach spaces. We
show in fact that our \ops\ enjoy a stronger property:

\begin{theorem}\label{th2}
If $X$ is a \sep\ complex Banach space which has an unconditional Schauder decomposition, then $X$ supports a bounded \op\ $T$ which is ergodic with
respect to a \nd\ \inv\ \ga\ \mea.
\end{theorem}

The Hilbert space construction is carried out in Section 2, and Theorems \ref{th1} and \ref{th2} are proved in Section 3. Finally, we
show in Section 4 that the situation is drastically different if we change the scalar field: there are real Banach spaces with an unconditional basis which support
no chaotic operators. The corresponding question for frequent \hy ity is still open.
\par\smallskip
Let us finish this introduction by mentioning that it is possible to extend Theorem \ref{th1} to the case of Fr\'{e}chet spaces \cite{DFGP_Frechet}: 
every complex separable Fr\'{e}chet space with a continuous norm and an unconditional Schauder decomposition admits a frequently hypercyclic and
chaotic operator. The same result holds true for complex Fr\'{e}chet spaces with an unconditional basis. 

\section{Some frequently hypercyclic and chaotic operators on Hilbert spaces}
So our aim is to show that if $X$ is a separable complex infinite-dimensional Banach space with an unconditional Schauder decomposition, then $X$
supports an \op\ which is both chaotic and frequently hypercyclic. We begin by constructing a particular class of nuclear perturbations of diagonal
\ops\ on a Hilbert space, which consists of \fhy\ and chaotic \ops.
\par\smallskip
For $n\geq 0$, let $H_{n}$ be the space $\ell^{2}$ endowed with the canonical basis $(e_{i,n})_{i\geq 0}$, and let $H=\oplus_{\ell^{2}}H_{n}$ be the
orthogonal sum of all the spaces $H_{n}$. If $(\mu_{n})_{n\geq 0}$ is any bounded sequence of complex numbers, the diagonal \op\ $\bar{D}_{\mu}$ is
defined by $\bar{D}_{\mu}(\oplus x_{n})=\oplus \mu_{n}x_{n}$ for any element $x=\oplus x_{n}$ in $H$. Now let $((w _{i,n})_{i\geq 0})_{n\geq 0}$ be a
sequence of positive weights such that
$$
\sup_{n\geq 0}\;\sup_{i\geq 0}w _{i,n}<+\infty .
$$
We define the \op\ $\bar{B}_{w }$ on $H$ by setting
$$
\bar{B}_{w }x=\sum_{n\geq 1}\sum_{i\geq 0}\pss{x}{e_{i,n}} \,w _{i,n-1} e_{i,n-1}.
$$
This is clearly a bounded \op\ on $H$, which is nothing but a backward weighted shift:
 $\bar{B}_{w }e_{i,0}=0$ for any $i\geq 0$, and
 $\bar{B}_{w }e_{i,n}=w _{i,n-1}e_{i,n-1}$ for $n\geq 1$ and $i\geq 0$.

\begin{theorem}\label{th3}
 Let $w =((w _{i,n})_{i\geq 0})_{n\geq 0}$ be a bounded sequence of positive weights.
 There exists a sequence $(\mu_{n})_{n\geq 0}$ of unimodular numbers such that the \op\
 $\bar{T}=\bar{D}_{\mu }+\bar{B}_{w }$ is \fhy\ and chaotic on $H$.
\end{theorem}

The proof of Theorem \ref{th3} relies on a criterion for frequent hypercyclicity which was proved in \cite{BG1} in the Hilbert space setting, and which
states that if $T\in \mathcal{B}(H)$ has ``sufficiently many'' \eve s associated to \eva s of modulus $1$, then $T$ is \fhy. Here is the precise
definition:

\begin{definition}\label{def1}
We say that a bounded operator $T$ on $X$ has a \ps\ if there exists a continuous probability measure $\sigma  $ on the unit circle $\T$ such that for
every $\sigma  $-measurable subset $A$ of $\T$ which is of $\sigma  $-measure $1$, $\spa \bigcup_{\lambda \in A} \ker (T-\lambda I) $ is dense in $X$.
\end{definition}
In other words if we take out from the unit circle a set of $\sigma  $-measure $0$ of \eva s, the \eve s associated to the remaining \eva s still span
$X$.
\par\smallskip
The following result is proved in \cite{BG2}:

\begin{theorem}\cite{BG2}\label{th0}
If $T$ is a bounded \op\ acting on a \sep\ infinite-dimensional complex Hilbert space $H$, and if $T$ has a \ps, then $T$ is \fhy.
\end{theorem}

See \cite{Grivaux10} for a different proof which extends to the Banach space case. We are now ready for the proof of Theorem \ref{th3}.

\begin{proof}[Proof of Theorem \ref{th3}]
 Let $\mu =(\mu _{n})_{n\geq 0}$ be for the moment an arbitrary sequence of
 unimodular numbers. Let $\lambda \in\T$, and $x=\oplus x_{n}\in H$.
 Then $\bar{T}x=\lambda x$ \ifff\ for any $n\geq 0$ and any $i\geq 0$,
$$
\mu _{n}x_{n}+\sum_{i\geq 0}\pss{x_{n+1}}{e_{i,n+1}}w _{i,n}e_{i,n}=\lambda x_{n},
$$
that is
$$
\pss{x_{n}}{e_{i,n}}=\frac{\lambda -\mu_{n-1}}{w _{i,n-1}}\pss{x_{n-1}}{e_{i,n-1}} \quad \textrm{so that}\quad
\pss{x_{n}}{e_{i,n}}=\prod_{p=0}^{n-1}\left( \frac{\lambda -\mu_{p}}{w _{i,p}}\right).
$$
Hence provided the series involved are convergent, the vectors
$$
E_{i}(\lambda )=e_{i,0}+\sum_{n\geq 1} \left(\prod_{p=0}^{n-1} \frac{\lambda -\mu_{p}}{w _{i,p}}\right) e_{i,n}, \quad i\geq 0
$$
are \eve s of $\bar{T}$ which span the eigenspace $\ker(\bar{T}-\lambda )$. Now our aim is to construct the diagonal coefficients $\mu _{p}$, 
 and a Cantor set $K\subset \T$ containing the $\mu_p$'s, in such a way
that the eigenvector fields $E_i$, $i\geq 0$, are well-defined and continuous on $K$, and the 
unimodular \eve s of $\bar{T}$ are perfectly spanning.
\par\smallskip
Write the set $\N_0$ of nonnegative integers as the disjoint union of successive intervals $J_{k}$, $k\geq 0$, where $|J_{k}|=2^{k-1}$ for $k\geq 1$:
$J_{0}=\{0\}$, $J_{1}=\{1\}$, $J_{2}=\{2,3\}$, $J_{3}=\{4,5,6,7\}$, $J_{4}=\{8,9,10,11,12,13,14,15\}$, etc... and more generally
$J_{k}=\{2^{k-1},\ldots, 2^{k}-1\}$.
\par\smallskip
$\bullet$ \textbf{Step 0:} to begin with, we take $\mu_{0}=1$.
\par\smallskip
$\bullet$ \textbf{Step 1:} let $\mu _{1}$ be an element of $\T$ distinct from $\mu _{0}$ with $|\mu _{0}-\mu _{1}|<1$ and such that the length
$l_{1}=|\mu _{0}-\mu _{1}|$ of the closed arc $\Gamma_1 $ joining in $\T$ $\mu _{0}$ and $\mu _{1}$ is so small that
$$l_{1}^{2}\sum_{n=2}^{3}\left(\prod_{p=0}^{n-1}\frac{1}{w_{i,p}^{2}}\right)<2^{-3}\quad \textrm{ for } i=0,1.$$
\par\smallskip
$\bullet$ \textbf{Step k:}  at step $k$, we construct the unimodular numbers $\mu _{n}$ for $n\in J_{k}$.
We choose  them in the set $  \bigcup_{p\in J_{k-1}} \Gamma_p$, where, for $p\in J_{k-1}$, $\Gamma  _{p}$ denotes the closed arc in $\T$  joining $\mu _{p}$ and $ \mu _{p-2^{k-2}}$. We also require that
$\mu _{0},\ldots,\mu _{2^{k}-1}$ be all distinct, and that for each $j=0,\ldots,2^{k-1}-1$, $\mu _{2^{k-1}+j}$ be very
close to $\mu _{j}$. More precisely, we pick these numbers $\mu _{2^{k-1}},\ldots, \mu _{2^{k}-1}$ in such a way that $l_{k}=\max\{|\mu _{2^{k-1}+j}-\mu _{j}| \textrm{ ; }
j=0,\ldots,2^{k-1}-1\}$ is so small that
$$
l_{k}^{2}\sum_{n\in J_{k+1}}\left(\prod_{p=0}^{n-1}\frac{1}{w_{i,p}^{2}}\right)<2^{-(k+2)} 
\quad \textrm{ for } i=0,1, \ldots, k.
$$
We then define $\Gamma_{2^{k-1}+j}$  for $j=0,\dots , 2^{k-1}-1$ as the closed arc in $\T$ joining $\mu_j$ and $\mu_{2^{k-1}+j}$.
\par\smallskip
Let $K$ be the compact subset of $\T$ defined by
$$
K:=\bigcap_{k\in\N} \left( \bigcup_{j\in J_k} \Gamma_j \right) .
$$
By construction, $\mu_p$ is in $ K$ for every $p\geq 0$. Fix $\lambda \in K$. Since 
 $\lambda $ belongs to $\bigcup_{j\in J_{k}} \Gamma_j$ for any $k\geq 1$,  there exists for any $k\geq 1$ an integer $q\in J_k$ with $|\lambda -\mu_q|\leq l_k$. Moreover the diameter of $K$ is smaller than $1$, and thus for every $k\geq 1$ and every $n\geq 2^k$, we have
$$
\prod_{p=0}^{n-1} |\lambda-\mu_p|^2 < l_k^2.
$$
Therefore we have for  each $i\geq 0$ and each $k'\geq i$ 
$$
\sum_{k\geq k'} \sum_{n\in J_{k+1}}\left(\prod_{p=0}^{n-1}\frac{|\lambda-\mu_p|^2}{w_{i,p}^{2}}\right) <
\sum_{k\geq k'} \sum_{n\in J_{k+1}}\left( l_{k}^{2}\prod_{p=0}^{n-1}\frac{1}{w_{i,p}^{2}}\right)<\sum_{k\geq k'} 
{2^{-(k+2)}} ={2^{-(k'+1)}}\cdot
$$
This implies that the series defining $E_i(\lambda )$ converges for every $\lambda \in K$ and every $i\geq 0$. 
\par\smallskip
Let us now prove that the \eve\ field $E_{i}$ is continuous on $K$ for each $i\geq 0$.
For any $\varepsilon>0$, let $k_0\geq i$ be such that $2^{-k_0}<\varepsilon /2$, and let
$$P(\lambda )=e_{i,0}+\sum_{n= 1}^{2^{k_{0}}} \left(\prod_{p=0}^{n-1} \frac{\lambda -\mu_{p}}{w _{i,p}}\right) e_{i,n}$$ be the polynomial in $\lambda $ corresponding to the $2^{k_0}$-th partial sum in the 
expression of $E_i$.
Then for all $\lambda, \lambda'\in K$ we have
$$
\norm{E_i(\lambda )-E_i(\lambda ')}^2<\norm{P(\lambda )-P(\lambda ')}^2+2.{2^{-{(k_0+1)}}}<\norm{P(\lambda )-P(\lambda ')}^2+\frac{\varepsilon }{2}\cdot
$$  Let now $\delta>0$ be such that $\norm{P(\lambda )-P(\lambda ')}^2<\varepsilon/2$ whenever 
$|\lambda-\lambda'|<\delta$. Then 
$$
\norm{E_i(\lambda )-E_i(\lambda ')}^2<\varepsilon , 
$$
 for any $\lambda,\lambda'\in K$ with $|\lambda-\lambda'|<\delta$, and this proves that the eigenvector fields 
 $E_i$, $i\geq 0$, are continuous on $K$. 
\par\smallskip
Observe now that the construction is done in such a way that all the coefficients $\mu_{p}$ are distinct. This implies that the \eve s $E_{i}(\lambda
)$, $i\geq 0$, $\lambda \in K$, span a dense subspace of $H$. Indeed for any $N\geq 0$, $\mu _{N}$ belongs to $K$ by construction, and
$$
E_{i}(\mu _{N})=e_{i,0}+\sum_{n= 1}^{N}\left(  \prod_{p=0}^{n-1}\frac{\mu_{N} -\mu_{p}}{w _{i,p}}\right) e_{i,n}.
$$
Hence $\textrm{sp} [E_{i}(\mu _{N}) \textrm{ ; } {i, N\geq 0}]$ contains all the vectors $e_{i,n}$, $i\geq 0 ,n\geq 0$. By continuity of the
eigenvector fields $E_i$, if $\sigma  $ is any continuous measure whose support is $K$, the \eve s $E_{i}(\lambda )$, $i\geq 0$, $\lambda \in K$, are
perfectly spanning \wrt\ $\sigma  $. Such a measure $\sigma$ exists because $K$ is a perfect compact subset of $\T$. Since our \op\ $\bar{T}$ is living
on a Hilbert space, Theorem \ref{th0} can be applied and $\bar{T}$ is \fhy.
\par\smallskip
If we additionally require that $\bar{T}$ be chaotic, it suffices to choose all the coefficients $\mu_{p}$ to be $n^{th}$ roots of $1$. Then the \eve
s of $\bar{T}$ associated to \eva s which are $n^{th}$ roots of $1$ span a dense subspace of $H$, and $\bar{T}$ is chaotic.
\end{proof}

\section{Proofs of Theorems \ref{th1} and \ref{th2}}

\subsection{Unconditional Schauder decompositions}
Let $X$ be a separable infinite-dimensional Banach space which admits an unconditional Schauder decomposition. This means that there exists a sequence
$(X_{n})_{n\geq 0}$ of closed subspaces of $X$ such that any $x\in X$ can be written in a unique way as an unconditionally convergent series
$x=\sum_{n\geq 0}x_{n}$, where $x_{n}$ belongs to $X_{n}$ for any $n\geq 0$. We denote in this case by $P_{n}$ the canonical projection $x\mapsto
x_{n}$ of $X$ onto $X_{n}$. If $(X_{n})_{n\geq 0}$ is an unconditional Schauder decomposition of $X$, and $(I_{k})_{k\geq 0}$ is any partition of $\N$
into finite or infinite subsets, let $Y_{k}$ denote the closed linear span of the spaces $X_{n}$, $n\in I_{k}$. Then $(Y_{k})_{k\geq 0}$ is also an
unconditional Schauder decomposition of $X$. Hence we will always suppose in the sequel that $(X_{n})_{n\geq 0}$ is an unconditional Schauder
decomposition of $X$ with all the subspaces $X_{n}$ infinite-dimensional. If $(\mu_{n})_{n\geq 0}$ is any bounded sequence of complex numbers, then the
multiplication operator $D_{\mu}:X\rightarrow X$ defined by
$$D_{\mu}(\sum_{n\geq 0}x_{n})=\sum_{n\geq 0}\mu_{n}x_{n}$$ is a
bounded operator on $X$ since the decomposition $X=\oplus_{n\geq 0}X_{n}$ is unconditional. Now since all $X_{n}$'s are infinite-dimensional, each of
them admits a biorthogonal system $(x_{i,n},x^{*}_{i,n})_{i\geq 0}$, where $x_{i,n}\in X_{n}$, $x_{i,n}^{*}\in X_{n}^{*}$ and $\pss {x_{i,n}^{*}}
{x_{j,n}} =\delta _{ij}$ for $i,j\geq 0$. Since $X=X_{n}\oplus \overline{\spa}\bigcup_{p\neq n}X_{p} $, we can extend $x_{i,n}^{*}$ to $X$ by setting
$x^{*}_{i,n}=0$ on $\overline{\spa}\bigcup_{p\neq n}X_{p} $. For $n\geq 1$, let $(w _{i,n})_{i\geq 0}$ be a sequence of positive weights going to zero
very fast when $i$ goes to infinity. Denote by $w $ the collection of sequences $((w _{i,n})_{i\geq 0})_{n\geq 1}$, and define the \op\ $B_{w }$ on $X$
by setting

$$
B_{w }x=\sum_{n\geq 1}\sum_{i\geq 0}\pss{x_{i,n}^{*}}{x} \,w _{i,n} x_{i,n-1} =\sum_{n\geq 1}\sum_{i\geq 0}\pss{x_{i,n}^{*}}{P_{n}x} \,w _{i,n}
x_{i,n-1}.
$$
Let $\norm{\cdot}^*$ be the dual norm of $\norm{\cdot}$. If the series

$$
\sum_{n\geq 1}\sum_{i\geq 0}w _{i,n} \,||x_{i,n-1}||\,.\, ||x_{i,n}^{*}||^*
$$
is convergent (which is the case if the quantities $w _{i,n}$ are suitably small), $B_{w }$ is a nuclear operator, hence a bounded \op.

\par\smallskip

\subsection{Proof of Theorem \ref{th1}}
The proof of Theorem \ref{th1} is now a straightforward application of the transference principle, as applied for instance in \cite{A}. Without loss of
generality we can suppose that

$$
\sum_{n\geq 0}\left(\sum_{i\geq 0}||x_{i,n}||^{2}\right)<+\infty .
$$
Then the \op\ $J:H\longrightarrow X$ defined by $Je_{i,n}=x_{i,n}$ for $i,n\geq 0$ is bounded, injective, and has dense range. Let $((w _{i,n})_{i\geq
0})_{n\geq 0}$ be a bounded sequence of positive weights such that the \op\ $B_{w }$ defined above is bounded on $X$. If $\bar{B}_{w }$ denotes the
backward weighted shift on $H$ associated to $w $, then $J\bar{B}_{w }=B_{w }J$. Now by Theorem \ref{th3} there exists a sequence of unimodular numbers
such that  $\bar{T}=\bar{D}_{\mu }+\bar{B}_{w }$ is \fhy\ and chaotic on $H$. Then $J(\bar{D}_{\mu }+\bar{B}_{w })=(D_{\mu }+B_{w })J$, and since $J$
is injective and has dense range, $T$ is \fhy\ and chaotic on $X$.

\subsection{Ergodicity with respect to an invariant gaussian measure: proof of Theorem \ref{th2}}
The proof of Theorem \ref{th2} is an immediate consequence of the intertwining equation $J(\bar{D}_{\mu }+\bar{B}_{w })=(D_{\mu }+B_{w })J$ above.
Since $\bar{T}$ acts on a Hilbert space, $\bar{T}$ admits an ergodic \nd\ \inv\ \ga\ \mea\ $\bar{m}$ by \cite{BG2}. Let $m$ be the \ga\ \mea\ on $X$
defined by $m(A)=\bar{m}(J^{-1}(A))$ for any Borel subset $A$ of $X$. This \mea\ is \nd\ and invariant by $T$. Lastly it is not difficult to check that
$T$ is \erg\ \wrt\ $m$: if $m(A)>0$ and $m(B)>0$, there exists an integer $N$ such that $\bar{m}(\bar{T}^{-N}(J^{-1}(A))\cap J^{-1}(B))>0$. Since
$J^{-1}(T^{-N}(A)\cap B)=\bar{T}^{-N}(J^{-1}(A))\cap J^{-1}(B)$ by the intertwining equation $TJ=J\bar{T}$, we have $m(T^{-N}(A)\cap B)>0$. This shows
that $T$ is \erg\ \wrt\ $m$.

\section{Chaotic operators and unconditional basis: the real
case}\label{real}

In this section, we observe that the situation changes completely if instead of considering complex Banach spaces as in Section 3, we consider real
Banach spaces.
\par\smallskip
The counterexample of Theorem \ref{real_gowers} below is built on a real Banach space $X_G$ constructed by Gowers in \cite{gowers1994a}: $X_{G}$ has an
unconditional basis $(e_{n})_n$, and  every  bounded operator $T\in L(X_G)$ is of the form $T=D_a+S$, where $D_a$ is a diagonal operator \wrt\ the
basis $(e_{n})_n$ associated to a bounded weight $a=(a_n)_n \in \R^\N$, and $S$ is a strictly singular operator (see \cite{gowers_maurey1997banach}).

\begin{theorem}\label{real_gowers}
The real separable infinite-dimensional Banach space $X_G$ has an unconditional basis, but admits no chaotic operator.
\end{theorem}

\begin{proof}
Suppose that  $T\in \mathcal{B}(X_G)$ is a chaotic operator on $X_{G}$, $T=D_a+S$, where $D_a$ is the diagonal operator \wrt\ the unconditional basis
$(e_{n})_n$ corresponding to $a=(a_n)_n \in \R^\N$ . Then its complexification $\wt{T}$ on the complexification $\wt{X}_{G}$ is also chaotic
\cite{bes_peris1999hereditarily}, and $\wt{T}$ can be written as $\wt{T}=\wt{D}_a+\wt{S}$. Since $\wt{T}$ is chaotic, its spectrum has no isolated
points and intersects $\T\setminus \R$. Therefore we can select a boundary point $\lambda$ of  $\sigma (\wt{T})\setminus\R$ and, by Putnam's theorem
(see, e.g., Proposition 3.7.8 in \cite{LN}) we obtain that $\lambda$ belongs to the essential spectrum of $\wt{T}$. On the other hand, since every
$a_n$ belongs to $\R$, the operator $\wt{D}_a-\lambda I$ is invertible on $\wt{X}_{G}$. Thus $\wt{T}-\lambda I$, being a perturbation of an invertible
operator by a strictly singular operator, is a Fredholm operator of index $0$ on $\wt{X}_{G}$, which contradicts the fact that $\lambda\in
\sigma_e(\wt{T})$. Hence $T$ is not chaotic on $X_{G}$.
\end{proof}

Theorem \ref{real_gowers} leads naturally to the following question:

\begin{question}
 Does the space $X_{G}$ support a \fhy\ \op? More generally, does there exist a real \sep\ Banach space with an unconditional basis which supports no
 \fhy\ \op?
\end{question}

\par\bigskip

\textit{Acknowledgement:} We would like to thank P. Tradacete for interesting discussions about the results in Section~\ref{real}.

\end{document}